\documentclass{article}
\usepackage{amsthm,amssymb,amsmath,url}

\title{Spectrally degenerate graphs: Hereditary case}

\author{Zden\v{e}k Dvo\v{r}\'ak\thanks{Supported in part
   by the grant GA201/09/0197 of Czech Science Foundation.}\\
  Institute for Theoretical Computer Science (ITI)\thanks{Institute for Theoretical Computer Science is supported as
              project 1M0545 by the Ministry of Education of the Czech Republic.}\\
  Charles University\\
  Prague, Czech Republic\\
  email: {\tt rakdver@kam.mff.cuni.cz}
\and
  Bojan Mohar\thanks{Supported in part by the
  Research Grant P1--0297 of ARRS (Slovenia), by an NSERC Discovery Grant (Canada)
  and by the Canada Research Chair program.}~\thanks{On leave from:
  IMFM \& FMF, Department of Mathematics, University of Ljubljana, Ljubljana,
  Slovenia.}\\
  {Department of Mathematics}\\
  {Simon Fraser University}\\
  {Burnaby, B.C. V5A 1S6} \\
  email: {\tt mohar@sfu.ca}
}

\newtheorem{theorem}{Theorem}[section]
\newtheorem{lemma}[theorem]{Lemma}

\newcommand{\DEF}[1]{{\em #1\/}}
\newcommand{\RR}{\ensuremath{\mathbb R}}

\newcommand{\D}{\ensuremath{\Delta}}
\newcommand{\DG}{\ensuremath{\Delta(G)}}
\begin{document}

\maketitle

\begin{abstract}
It is well known that the spectral radius of a tree whose maximum degree
is $\D$ cannot exceed $2\sqrt{\D-1}$. A similar upper bound holds for arbitrary
planar graphs, whose spectral radius cannot exceed $\sqrt{8\D}+10$, and
more generally, for all $d$-degenerate graphs, where the corresponding 
upper bound is $\sqrt{4d\D}$.
Following this, we say that a graph $G$ is {\em spectrally $d$-degenerate}
if every subgraph $H$ of $G$ has spectral radius at most $\sqrt{d\D(H)}$.
In this paper we derive a rough converse of the above-mentioned results
by proving that each spectrally $d$-degenerate graph $G$
contains a vertex whose degree is at most $4d\log_2(\DG/d)$ (if $\DG\ge 2d$).
It is shown that the dependence on $\Delta$ in this upper bound cannot be eliminated,
as long as the dependence on $d$ is subexponential.
It is also proved that the problem of deciding if a graph is spectrally
$d$-degenerate is co-NP-complete.
\end{abstract}

\section{Introduction}

All graphs in this paper are finite and \DEF{simple}, i.e.\ no loops or multiple
edges are allowed. We use standard terminology and notation.
We denote by $\Delta(G)$ and $\delta(G)$ the maximum and the minimum degree of $G$,
respectively. If $H$ is a subgraph of $G$, we write $H\subseteq G$.
For a graph $G$, let $\rho(G)$ denote its \DEF{spectral radius}, the largest eigenvalue of
the adjacency matrix of $G$. More generally, if $M$ is a square matrix,
the spectral radius of $M$, denoted by $\rho(M)$, is the maximum modulus
$|\lambda|$ taken over all eigenvalues $\lambda$ of $M$.

If $T$ is a tree, then it is a subgraph of the infinite $\D(T)$-regular tree.
This observation implies that the spectral radius of $T$
is at most $2\sqrt{\D(T)-1}$. Similar bounds have been obtained for arbitrary
planar graphs and for graphs of bounded genus \cite{DM}. In particular,
the following result holds.

\begin{theorem}[Dvo\v r\'ak and Mohar \cite{DM}]
\label{thm:planar}
If\/ $G$ is a planar graph, then $$\rho(G)\le \sqrt{8\DG} + 10.$$
\end{theorem}

The proof in \cite{DM} uses the fact that the edges of every planar graph
$G$ can be partitioned into two trees of maximum degree at most
$\DG /2$ and a graph whose degree is bounded by a small constant.
A similar bound was obtained earlier by Cao and Vince~\cite{CV}.

Whenever a result can be proved for tree-like graphs and for graphs of
bounded genus, it is natural to ask if it can be extended to a more general
setting of minor-closed families. Indeed, this is possible also in our case,
and a result of Hayes~\cite{Hayes} (see Theorem \ref{thm:Hayes} below)
goes even further.

A graph is said to be \DEF{$d$-degenerate} if every subgraph of $G$ has
a vertex whose degree is at most $d$. This condition is equivalent to
the requirement that $G$ can be reduced to the empty graph by successively
removing vertices whose degree is at most~$d$.

A requirement that is similar to degeneracy is existence of an orientation of
the edges of $G$ such that each vertex has indegree at most $d$. Every such graph
is easily seen to be $2d$-degenerate, and conversely, every $d$-degenerate graph
has an orientation with maximum indegree $d$.

\begin{theorem}[Hayes~\cite{Hayes}]
\label{thm:Hayes}
Any graph $G$ that has an orientation with maximum indegree $d$
(hence also any $d$-degenerate graph) and with $\D=\DG\ge 2d$
satisfies $$\rho(G)\le 2\sqrt{d(\D-d)}.$$
\end{theorem}

It is well-known that each planar graph $G$ has an orientation with maximum
indegree $3$. Theorem \ref{thm:Hayes} thus implies that
$\rho(G)\le \sqrt{12(\D-3)}$, which is slightly weaker than the bound of
Theorem \ref{thm:planar} (for large $\D$).

The above results suggest the following definitions. We say that a graph $G$
is \DEF{spectrally $d$-degenerate} if every subgraph $H$ of $G$ has
spectral radius at most $\sqrt{d\D(H)}$. Hayes' Theorem \ref{thm:Hayes}
shows that $d$-degenerate graphs are spectrally $4d$-degenerate.
The implication is clear for graphs $G$ of maximum degree at least $2d$.
On the other hand, if $\DG\le2d$, then $\rho(G)\le \DG\le \sqrt{2d\DG}$.
The main result of this paper is a rough converse of this statement.

\begin{theorem}
\label{thm:almost degenerate}
If\/ $G$ is a spectrally $d$-degenerate graph, then it
contains a vertex whose degree is at most $\max\{4d,4d\log_2(\DG/d)\}$.
\end{theorem}

The proof is given in Section \ref{sect:main}.
If it were not for the annoying factor of
$\log(\D)$, this would imply $f(d)$-degeneracy, which was our initial
hope. However, in Section~\ref{sect:LB} we construct examples
showing that the ratio between degeneracy and spectral degeneracy
may be as large as (almost) $\log\log \DG$.

In the last section, we consider computational complexity questions 
related to spectral degeneracy. First we prove that for every integer
$d\ge 3$, it is NP-hard to decide if the spectral degeneracy of a given
graph $G$ of maximum degree $d+1$ is at least $d$. Next we show that
the problem of deciding if a graph is spectrally $d$-degenerate is 
co-NP-complete.

\section{Spectral radius}

We refer to \cite{Bi,CDS,GoRo} for basic results about the spectra of finite graphs
and to \cite{HJ} for results concerning the spectral radius of (nonnegative) matrices.
Let us review only the most basic facts that will be used in this paper.
The spectral radius is monotone and subadditive. Formally this is stated in the following lemma.

\begin{lemma}
\label{lem:1}
{\rm (a)} If $H\subseteq G$, then $\rho(H)\le \rho(G)$.

{\rm (b)} If $G = K\cup L$, then $\rho(G)\le \rho(K) + \rho(L)$.
\end{lemma}

The application of Lemma \ref{lem:1}(a) to the subgraph of $G$ consisting of
a vertex of degree $\Delta(G)$ together with all its incident edges gives
a lower bound on the spectral radius in terms of the maximum degree.

\begin{lemma}
\label{lem:2}
$\sqrt{\DG} \le \rho(G)\le \DG$.
\end{lemma}

A partition $V(G) = V_1\cup \cdots \cup V_k$ of the vertex set of $G$ is called
an \DEF{equitable partition} if, for every $i,j\in \{1,\dots,k\}$, there exists an integer $b_{ij}$ such that every vertex $v\in V_i$ has precisely $b_{ij}$
neighbors in $V_j$. The $k\times k$ matrix $B=[b_{ij}]$ is called
the \DEF{quotient adjacency matrix} of $G$ corresponding to this equitable
partition.

\begin{lemma}
\label{lem:3}
Let\/ $B$ be the quotient adjacency matrix corresponding to an equitable
partition of\/ $G$. Then every eigenvalue of $B$ is also an eigenvalue of $G$,
and $\rho(G)=\rho(B)$.
\end{lemma}

\begin{proof}
The first claim is well known (see \cite{GoRo} for details).
To prove it, one just lifts an eigenvector $y$ of $B$ to an eigenvector
$x$ of $G$ by setting $x_v = y_i$ if $v\in V_i$. By the Perron-Frobenius
Theorem, the eigenvector corresponding to the largest eigenvalue of $B$
is positive (if $G$ is connected, which we may assume), so its lift
is also a positive eigenvector of $G$. This easily implies (by using the
Perron-Frobenius Theorem and orthogonality of eigenvectors of $G$)
that this is the eigenvector corresponding
to the largest eigenvalue of $G$. Thus, $\rho(G)=\rho(B)$.
\end{proof}

We will need an extension of Lemma \ref{lem:3}. As above, let
$V(G) = V_1\cup \cdots \cup V_k$ be a partition of $V(G)$, and let
$n_i=|V_i|$, $1\le i\le k$. For every $i,j\in\{1,\dots,k\}$, let
$e_{ij}$ denote the number of ordered pairs $(u,v)$ such that $u\in V_i$,
$v\in V_j$ and $uv\in E(G)$, i.e.~$e_{ij}$ is the number of edges between
$V_i$ and $V_j$ if $i\ne j$, and is twice the number of edges between the
vertices in $V_i$ if $i=j$. Let $b_{ij} = e_{ij}/n_i$ and let $B=[b_{ij}]$
be the corresponding $k\times k$ matrix. This is a generalization from
equitable to general partitions, so we say that $B$ is
the \DEF{quotient adjacency matrix} of $G$ also in this case.
If a matrix $B'=[b'_{ij}]_{i,j=1}^k$ satisfies $0\le b'_{ij}\le b_{ij}$
for every pair $i,j$, then we say that $B'$ is a \DEF{quotient sub-adjacency matrix}
for the partition $V_1\cup \cdots \cup V_k$.

\begin{lemma}
\label{lem:4}
If\/ $B'$ is a quotient sub-adjacency matrix corresponding to a partition
of\/ $V(G)$, then $\rho(G)\ge\rho(B')$.
\end{lemma}

\begin{proof}
By the monotonicity of the spectral radius, $\rho(B')\le\rho(B)$, where $B$ is
the quotient adjacency matrix. So we may assume that $B'=B$.
The matrix $B$ is element-wise non-negative. By the Perron-Frobenius Theorem,
its spectral radius $\rho(B)$ is equal to the largest eigenvalue
of $B$ (which is real and positive) and the corresponding eigenvector
$y$ is non-negative.
Let us define the vector $f \in \RR^{V(G)}$ by setting $f_v=y_i$ if $v\in V_i$.
Then
$$
  \Vert f\Vert^2=\sum_{v\in V(G)}f_v^2
         =\sum_{i=1}^k n_i y_i^2.
$$
Furthermore, if $A$ is the adjacency matrix of $G$, then
\begin{eqnarray*}
  \langle f|Af\rangle &=& 2\sum_{uv\in E(H)} f_uf_v\\
      &=& \sum_{i=1}^k \sum_{j=1}^k e_{ij}y_iy_j \\
      &=& \sum_{i=1}^k n_iy_i \sum_{j=1}^k b_{ij}y_j \\
      &=& \rho(B)\sum_{i=1}^k n_iy_i^2 \\[1mm]
      &=& \rho(B)\Vert f\Vert^2.
\end{eqnarray*}
Since the matrix $A$ is symmetric, $\rho(A)$ is equal to the numerical radius
of $A$. Thus, it follows from the above calculations that
$\rho(G)\ge \frac{\langle f|Af\rangle}{\Vert f\Vert^2} = \rho(B)$,
which we were to prove.
\end{proof}

%
%

\section{Spectrally degenerate graphs are nearly degenerate}

\label{sect:main}

In this section we prove our main result, Theorem \ref{thm:almost degenerate}.
For convenience we state it again (in a slightly different form).

\begin{theorem}
\label{thm:almost degenerate1}
Let\/ $G_0$ be a spectrally $d$-degenerate graph with $r=\delta(G_0)>4d$.
Then $r \le 4d\log_2(\D(G_0)/d)$.
\end{theorem}

\begin{proof}
Suppose for a contradiction that $r > 4d\log_2(\D(G_0)/d) \ge 4d$.
Let $G$ be a subgraph of $G_0$ obtained by successively deleting edges $xy$
for which $\deg(x)\ge\deg(y)>r$, as long as possible. Then $G$ has the following
properties:
\begin{itemize}
\item[(a)] $\delta(G) = r > 4d\log_2(\D(G_0)/d) \ge 4d\log_2(\DG/d)$.
\item[(b)] $G$ is spectrally $d$-degenerate.
\item[(c)] The set of vertices of $G$ whose degree is bigger than $r$ is
an independent vertex set in $G$.
\end{itemize}
Our goal is to prove that $r \le 4d\log_2(\DG/d)$. This will contradict
(a) and henceforth prove the theorem.

Let us consider the vertex partition into the following vertex sets:
$$V_0 = \{v\in V(G)\mid \deg_G(v)=r\},$$
and for $i=1,\dots,l$,
$$V_i = \{v\in V(G)\mid 2^{i-1}r < \deg_G(v) \le 2^i r\},$$
where $l = \lceil \log_2(\DG/r) \rceil \le \log_2(\DG/d)$.
Let $B = [b_{ij}]_{i,j=0}^l$ be the quotient adjacency matrix for the
partition $V_0,V_1,\dots,V_l$ of $V(G)$. Since all vertices in $V_0$ have
the same degree $r$, it follows from the definitions of the entries of
$B$ that $r = \sum_{i=0}^l b_{0i}$. Thus it suffices to estimate the
entries $b_{0i}$ in order to bound $r$.

For $i=0$, let $H\subseteq G$ be the induced subgraph of $G$ on $V_0$.
Since $G$ is spectrally $d$-degenerate, we have that
$\rho(H)\le\sqrt{d\D(H)}\le \sqrt{dr} \le \sqrt{r^2/4} = \tfrac{r}{2}$.
On the other hand, since $H$ has average degree $b_{00}$, we have
$\rho(H)\ge b_{00}$. Thus, $b_{00} \le \tfrac{r}{2}$.
This shows that $\sum_{i=1}^l b_{0i}=r - b_{00}\ge r/2$, and thus
it suffices to prove that
\begin{equation}
\sum_{i=1}^l b_{0i}\le 2d\log_2(\DG/d).
\label{eq:zero}
\end{equation}

From now on we let $B'$ be the matrix obtained from $B$ by setting the entry
$b_{00}'$ to be $0$. This is the quotient adjacency matrix of
the subgraph $G'$ of $G$ obtained
by removing edges between pairs of vertices in $V_0$.

We shall now prove that
\begin{equation}
  \sum_{i=1}^t 2^{i-1}b_{0i}\le 2^t d
\label{eq:first}
\end{equation}
for every $t=1,\dots,l$. Let us consider the subgraph $G_t$ of $G'$
induced on $V_0\cup V_1\cup \cdots\cup V_t$ and the corresponding matrix
$$
   B_t = \left[\begin{matrix}
            0 & b_{01} & \dots & b_{0t} \\
            r & 0 & \dots & 0 \\
            2r & 0 & \dots & 0 \\
            4r & 0 & \dots & 0 \\
            \vdots & \vdots & \ddots & \vdots \\
            2^{t-1}r & 0 & \dots & 0
         \end{matrix}\right]
$$
Let us observe that the entries $2^{i-1}r$ ($i=1,\dots,t$) in the first column
of $B_t$ are smaller than the corresponding entries in $B'$ because every
vertex in $V_i$ has degree more than $2^{i-1}r$. Therefore, $B_t$ is a
quotient sub-adjacency matrix for the subgraph $G_t$.
By expanding the determinant of the matrix $\lambda I - B_t$, it is easy to
see that
\begin{equation}
  \rho(B_t)^2 = \sum_{i=1}^t 2^{i-1}r b_{0i}.
  \label{eq:2a}
\end{equation}
Using Lemma \ref{lem:4} and
the fact that $G_t$ is spectrally $d$-degenerate, we see that
$\rho(B_t)^2\le \rho(G_t)^2\le d\cdot2^t r$. This inequality combined with
(\ref{eq:2a}) implies~(\ref{eq:first}).

We shall now prove by induction on $s$ that
\begin{equation}
  \sum_{i=1}^s b_{0i}\le (s+1)d
\label{eq:firstx}
\end{equation}
for every $s=1,\dots,l$. For $s=1$, this is the same as the inequality (\ref{eq:first})
taken for $t=1$. For $s\ge2$, we apply inequalities (\ref{eq:first}) to get
the following estimates:
\begin{equation}
  2^{s-t}\sum_{i=1}^t 2^{i-1}b_{0i}\le 2^s d
\label{eq:second}
\end{equation}
and henceforth
\begin{equation}
  \sum_{t=1}^s 2^{s-t}\sum_{i=1}^t 2^{i-1}b_{0i}\le s \cdot 2^s d.
\label{eq:third}
\end{equation}
Finally, inequality (\ref{eq:second}) (taken with $t=s$) and (\ref{eq:third}) imply
\begin{eqnarray*}
 2^s\sum_{i=1}^s b_{0i}
   &=& \sum_{i=1}^s \Bigl( 2^{i-1} + \sum_{j=i}^s 2^{j-1} \Bigr) b_{0i} \\
   &=& \sum_{i=1}^s 2^{i-1} b_{0i} +
       \sum_{t=1}^s 2^{s-t}\sum_{i=1}^t 2^{i-1}b_{0i} \\[2.5mm]
   &\le& 2^s d + s \cdot 2^s d = 2^s(s+1)d.
\end{eqnarray*}
This proves (\ref{eq:firstx}).  For $s=l$, this implies (\ref{eq:zero}) and completes the proof of the theorem.
\end{proof}

\section{A lower bound}

\label{sect:LB}

In this section we show that the $\log(\D)$ factor in the bound given by Theorem~\ref{thm:almost degenerate} cannot be eliminated entirely.

Let $\alpha\in \RR_+$. We say that a graph $G$ is \DEF{$\alpha$-log-sparse},
shortly \DEF{$\alpha$-LS}, if every subgraph $H$ of $G$ has average degree at
most $\alpha \log(\D(H))$. Observe that being $\alpha$-LS is a hereditary
property and that every $\alpha$-LS graph $G$ is $\alpha\log(\DG)$-degenerate.

Pyber, R\"odl, and Szemer\'edi \cite[Theorem 2]{PRS} proved that there exists a
constant $\alpha_0$ such that every graph $G$ with average degree at least
$\alpha_0\log(\D(G))$ contains a 3-regular subgraph. On the other hand, they
proved in the same paper \cite{PRS} that there exists a constant $\beta>0$ such
that, for each $n\ge1$, there is a bipartite graph of order $n$ with average
degree at least $\beta \log\log n$ which does not contain any 3-regular
subgraph (and is hence $\alpha_0$-LS). These results establish the following.

\begin{theorem}[\cite{PRS}]
\label{thm:PRS}
There exist constants $\alpha_0,\beta_0>0$ such that for every integer
$\tau>1$ there exists a bipartite graph $G$ with bipartition $V(G)=A\cup B$ with the following properties:
\begin{itemize}
\item[\rm (a)] $G$ is $\alpha_0$-LS.
\item[\rm (b)] $|A|\ge |B|$ and every vertex in $A$ has degree $\tau$.
\item[\rm (c)] $\beta_0\log\log |A|\le \tau$.
\end{itemize}
\end{theorem}

We will prove that graphs of Theorem \ref{thm:PRS} have small spectral degeneracy.
The proof will use the Chernoff inequality in the following form (cf. \cite{mat}, Theorem 7.2.1)

\begin{lemma}\label{lem:chernoff}
Let $X_1, \ldots, X_n$ be independent random variables, each of them
attaining value $1$ with probability $p$, and having value\/ $0$ otherwise.
Let $X=X_1+\cdots+X_n$.  Then, for any $r>0$,
$$\mbox{{\rm Prob}}\bigl[\,|X-np|\ge r\,\bigr] < \exp\Bigl(-\frac{r^2}{2(np+r/3)}\Bigr).$$
\end{lemma}

We can now prove the following lemma, showing that a bipartite graph
whose bipartite parts are ``almost'' regular cannot be log-sparse.

\begin{lemma}
\label{lem:log-dense subgraph}
Let\/ $T\ge 10$ and\/ $t>0$ be integers such that $$6\alpha_0 \log(20T)\le t\le T.$$
Let\/ $H$ be a bipartite graph of maximum degree $\D\ge 2T$ with bipartition $V(H)=A\cup B$ satisfying the following properties:
\begin{itemize}
\item[\rm (a)] $t\le \deg v\le T$ for each vertex $v\in A$.
\item[\rm (b)] Each vertex $v\in B$ has degree at least $\D/2$.
\end{itemize}
Then $H$ is not $\alpha_0$-LS.
\end{lemma}

\begin{proof}
Choose a subset $A'$ of $A$ by selecting each element uniformly independently with probability $p=2T/\D$, and let
$H'$ be the subgraph of $H$ induced by $A'\cup B$.  The expected size of $A'$ is $a'=2T|A|/\D$.
Note that $T|A|\ge |E(H)|\ge |B|\D/2$, thus $a'\ge |B|$.  Furthermore, $|A|\ge \D/2$, and thus
$a'\ge T\ge 10$.  By Lemma~\ref{lem:chernoff},
$$\mbox{Prob}\bigl[\,|A'|\le \tfrac{1}{2}a'\,\bigr] < e^{-3a'/28}< \tfrac{1}{2}.$$
Therefore, we have $2|A'|\ge a'\ge |B|$ with probability greater than $\tfrac{1}{2}$.

Consider a vertex $v\in B$.  The expected degree of $v$ in $H'$ is between $T$ and $2T$, and
the probability that $v$ has degree greater than $2cT$ is less than $e^{-cT}$ for any $c\ge 10$,
by Lemma~\ref{lem:chernoff}.  Let $z=0$ if $\deg_{H'} v\le 20T$ and $z=\deg_{H'} v$ otherwise.
The expected value of $z$ is
\begin{eqnarray*}
\sum_{j=20T+1}^\infty Pr(\deg_{H'} v=j)j&=&\sum_{j=20T+1}^\infty\sum_{i=1}^j Pr(\deg_{H'} v=j)\\
&=&\sum_{i=1}^{20T}\sum_{j=20T+1}^\infty Pr(\deg_{H'} v=j)+\sum_{i=20T+1}^\infty\sum_{j=i}^\infty Pr(\deg_{H'} v=j)\\
&=&20T Pr(\deg_{H'} v>20T) + \sum_{i=20T+1}^\infty Pr(\deg_{H'} v\ge i)\\
&=&20T Pr(\deg_{H'} v>20T) + \sum_{i=20T}^\infty Pr(\deg_{H'} v>i)\\
&\le&20Te^{-10T}+\sum_{i=20T}^\infty e^{-i/2}.
\end{eqnarray*}
We conclude that the expected number of edges of $H'$ incident with vertices of degree greater
than $20T$ is less than
$$|B|\Bigl(20Te^{-10T}+\sum_{i=20T}^\infty e^{-i/2}\Bigr)<|B|(20T+3)e^{-10T}.$$
By Markov's inequality, it happens with positive probability that $H'$ has less than $2|B|(20T+3)e^{-10T}$ edges incident with
vertices of degree greater than $20T$ and that $2|A'|\ge |B|$.

Let us now fix a subgraph $H'$ with these properties.
Let $H''$ be the graph obtained from $H'$ by removing the vertices of degree greater than $20T$.
Clearly, $\D(H'')\le 20T$.  Also, $H''$ has at most $3|A'|$ vertices and more than
$$|A'|t - 2|B|(20T+3)e^{-10T}\ge |A'|(t-4(20T+3)e^{-10T})\ge \tfrac{1}{2}|A'|t$$
edges, thus the average degree of $H''$ is greater than $t/6$.  Since $t/6\ge \alpha_0 \log(20T)$,
this shows that $H$ is not $\alpha_0$-LS.
\end{proof}

\begin{theorem}
\label{thm:main example}
Suppose that a bipartite graph $G$ with bipartition $V(G)=A\cup B$ satisfies properties (a)--(c)
of Theorem \ref{thm:PRS}, where $\tau\ge 10$ and $6\alpha_0\log(20\tau)\le \tau$. Then $G$ is spectrally $d$-degenerate,
where $$d = 48(3+2\sqrt{2})\alpha_0 \, \log(20\tau).$$
\end{theorem}

\begin{proof}
Suppose for a contradiction that $H$ is a subgraph of $G$ with maximum degree $D=\D(H)$ whose spectral radius violates spectral $d$-degeneracy requirement,
\begin{equation}
    \rho(H) > \sqrt{dD}.
\label{eq:P1}
\end{equation}
We may assume that $H$ is chosen so that $D$ is minimum possible.
Since $G$ is $\alpha_0$-LS, the same holds for its subgraph $H$. In particular, $H$ is $\alpha_0\log(D)$-degenerate and hence $\rho(H) \le 2\sqrt{\alpha_0\log(D)\cdot D}$. By (\ref{eq:P1}) we conclude that
\begin{equation}
    4\alpha_0\log(D) > d.
\label{eq:P2}
\end{equation}
This implies, in particular, that
\begin{equation}
    D \ge 2\tau.
\label{eq:P3}
\end{equation}
Let $\gamma = (3-2\sqrt{2})/8$. Let us partition the edges of $H$ into three subgraphs, $H=H_0\cup H_1\cup H_2$, such that the following holds:
\begin{itemize}
\item[\rm (a)] Each vertex in $V(H_0)\cap B$ has degree in $H_0$ at least $D/2$.
\item[\rm (b)] Each vertex in $V(H_0)\cap A$ has degree in $H_0$ at least $\gamma d$.
\item[\rm (c)] $H_1$ is $\gamma d$-degenerate.
\item[\rm (d)] $\D(H_2) \le D/2$.
\end{itemize}
Such a partition can be obtained as follows.  Let $H_0$ be a minimal induced subgraph of $H$
such that $E(H)\setminus E(H_0)$ can be partitioned into graphs $H_1$ and $H_2$ satisfying the
conditions (c) and (d) and $V(H_0)\cap V(H_1)\cap A=\emptyset$.  We claim that $H_0$ satisfies (a) and (b).  Indeed, suppose that
$H_0$ violates (a).  Then, there exists a vertex $v\in V(H_0)\cap B$ of degree at most $D/2$.
Consider the graph $H'_2$ obtained from $H_2$ by adding all edges of $H_0$ incident with $v$.
Clearly, $\D(H'_2)\le D/2$, since $v$ has degree at most $D/2$ and all vertices in $A\cap V(H'_2)$
have degree at most $\tau\le D/2$ by (\ref{eq:P3}).  Thus, there exists a partition of $E(H)\setminus E(H_0-v)$
satisfying (c) and (d), which contradicts the minimality of $H_0$.  Similarly, suppose that
$H_0$ violates (b), so there exists $v\in V(H_0)\cap A$ of degree at most $\gamma d$.
Since $V(H_0)\cap V(H_1)\cap A=\emptyset$, $v\not\in V(H_1)$, and thus the graph $H'_1$ obtained
from $H_1$ by adding all edges of $H_0$ incident with $v$ is $\gamma d$-degenerate.
Furthermore, $V(H_0-v)\cap V(H'_1)\cap A=\emptyset$, so we again obtain a contradiction with
the minimality of $H_0$.

Suppose that $H_0\ne\emptyset$. Then we use properties (a)--(b) of $H_0$ and
apply Lemma \ref{lem:log-dense subgraph} to conclude that $H_0$ is not
$\alpha_0$-LS. This contradicts our assumption that $G$ is $\alpha_0$-LS and
shows that $H_0$ must be empty.

Thus, $H = H_1\cup H_2$. Since $H$ was selected as a subgraph violating spectral degeneracy with its maximum degree smallest possible,
we conclude that $H_2$ is spectrally $d$-degenerate. By applying Lemma \ref{lem:1}(b) and using Theorem \ref{thm:Hayes} on $H_1$, we obtain
\begin{eqnarray*}
  \rho(H) &\le& \rho(H_1)+\rho(H_2)\\
  &\le& \sqrt{4\gamma d \D(H_1)} + \sqrt{d \D(H_2)} \\
  &\le& \sqrt{4\gamma d D} + \sqrt{d D/2} \\
  &\le& \bigl(\sqrt{4\gamma} + \sqrt{1/2}\,\bigr)\sqrt{d D} = \sqrt{d D}.
\end{eqnarray*}
This contradicts (\ref{eq:P1}) and completes our proof.
\end{proof}

By Theorem~\ref{thm:PRS}, there exist constants $\beta$ and $n_0$ such that we can
apply Theorem~\ref{thm:main example} to graphs on $n$ vertices with $\tau\ge \beta\log\log n$,
for any $n\ge n_0$.  Then, $d=O(\log\log\log n)$, and thus the ratio between the degeneracy
and the spectral degeneracy is at least $\Omega(\log\log n/\log\log\log n)\ge \Omega(\log\log \D/\log\log\log \D)$.

Let us however remark that this does not exclude the possibility that the degeneracy is bounded
by a function of the spectral degeneracy.  Answering a question we posed in the preprint version
of this paper, Alon~\cite{alon} proved that that is not the case.

\begin{theorem}
For every $M$, there exist spectrally $50$-degenerate graphs with minimum degree at least $M$.
\end{theorem}

\section{Computational complexity remarks}

Our results raise the problem of how hard it is to verify spectral degeneracy
of a graph.

\begin{quote}
	{\sc Spectral Degeneracy Problem} \\
	{\sc Input}: A graph $G$ and a positive rational number $d$. \\
	{\sc Task}: Decide if $G$ is spectrally $d$-degenerate.
\end{quote}

\noindent
Below we prove that this problem is co-NP-complete.  
To demonstrate this, we need some preliminary
results.  First, we show that distinct roots of a polynomial
cannot be too close to each other. For a polynomial $p(x)=\sum_{i=0}^k a_ix^i$ with integer coefficients,
let $a(p)=\log \max_{0\le i\le k} |a_i|$.

\begin{lemma}
  \label{lemma-distroots}
  Let $p(x)$ be an integer polynomial of degree $k$.
  If $p(u)=p(v)=0$ and $u\neq v$, then $-\log |u-v|=O(k^3(a(p)+\log k))$.
\end{lemma}

\begin{proof}
Mahler~\cite{mahler} proved that if $y$ and $z$ are two roots
of a polynomial $s(x)$ of degree $d$, then $-\log |y-z|=O(-\log |D| + d\log d + da(s))$,
where $D$ is the discriminant of $s$.
To apply this result, we need to eliminate the roots of $p$ with multiplicity greater than one.
By Brown~\cite{brown}, there exists an integer polynomial
$q(x)$ that is a greatest common divisor of $p(x)$ and $p'(x)$ such that
$a(q)=O(k(a(p)+\log k))$.
Let $c$ be the leading coefficient of $q$ and let $r(x)=c^kp(x)/q(x)$.  
Note that $r(x)$ is an integer polynomial, all of whose roots are simple, 
$r(u)=r(v)=0$, and $a(r)=O(k^2(a(p)+\log k))$.
Since $r$ is an integer polynomial with simple roots, the absolute value of its 
discriminant is at least $1$. Using the afore-mentioned result of 
Mahler~\cite{mahler}, we conclude that $-\log |u-v|=O(k^3(a(p)+\log k))$.
\end{proof}

Cheah and Corneil~\cite{regsg} showed the following.
\begin{theorem}\label{thm-regsg}
For any fixed integer $d\ge 3$, determining whether a graph of maximum degree $d+1$
has a $d$-regular subgraph is NP-complete.
\end{theorem}

We need an estimate on the spectral radius of graphs where the vertices of maximum degree
are far apart.

\begin{lemma}\label{lemma-near}
Let $G$ be a graph of maximum degree $d+1$ such that the distance between
every pair of vertices of degree $d+1$ is at least three.  Then
$$\rho(G)\le \sqrt[3]{(d+1)(d^2+1)}.$$
\end{lemma}

\begin{proof}
We may assume that $G$ is connected, since the spectral radius of a graph
is the maximum of the spectral radii of its components.
We use the fact that $\rho(G)=\limsup_{n\to \infty} \sqrt[n]c_n$,
where $c_n$ is the number of closed walks of length $n$ starting at an arbitrary vertex $v$ of $G$.
For any vertex $z$ of degree $d+1$, $G$ contains at most $(d+1)[(d-1)d+(d+1)]=(d+1)(d^2+1)$
walks of length $3$ starting at $z$, including those whose second vertex is $z$ as well.
Similarly, the number of walks of length $3$ from a vertex of degree at most $d$
is at most $(d+1)d^2$.  We conclude that $c_n\le \left[(d+1)(d^2+1)\right]^{\lceil n/3\rceil}$,
and the claim follows.
\end{proof}

We will also use the following result which shows that the spectral radius
of a connected non-regular graph of maximum degree $d$ cannot be arbitrarily close to $d$.

\begin{lemma}[Cioab\u{a} \cite{Ci}]
\label{lemma-nearlw}
Let $G$ be a connected graph of maximum degree $d$ and with diameter $D$.  
If $G$ has a vertex of degree less than $d$, then 
$$\rho(G) < d - \frac{1}{D|V(G)|}.$$
\end{lemma}


We can now proceed with examining the complexity of spectral degeneracy computation.

\begin{lemma}
\label{lem-coNP}
The {\sc Spectral Degeneracy Problem} is in co-NP.
\end{lemma}

\begin{proof}
To verify that the spectral degeneracy of $G$ is
greater than $d$, guess a connected subgraph $H$ of $G$ (on $k\le|V(G)|$ vertices)
such that $\rho(H)>\sqrt{d\Delta(H)}=b$.  To prove that $H$ has this property, first
compute the characteristic polynomial $p(x)=\det(xI-M)$, where $M$ is the adjacency matrix of $H$.
Note that the absolute value of each coefficient of $p$ is at most $k!$
and that $p$ can be computed in polynomial time using, for example,
Le Verrier--Faddeev's algorithm~\cite{leverrier}.
Then, we need to show that $p$ has a real positive root greater than $b$.  This is the case if
$p(b)<0$ and this condition can be verified in a polynomial time, since $b$ is a square root of a rational number.
Hence, we may assume that $p(b)\ge 0$. Let us recall that $\rho(H)$ is a simple root of $p$. 
Hence, if $\rho(H) > b$, then there exists a root $y$ of $p$ such that
$b \le y < \rho(H)$ and $p(x)<0$ when $y<x<\rho(H)$.  To prove that $b < \rho(H)$, it suffices to guess
a value $x>b$ such that $p(x)<0$, say any value between $y$ and $\rho(H)$.  By Lemma~\ref{lemma-distroots},
$-\log (\rho(H) - y) = O(k^4\log k)$, and thus such a number $x$ can be expressed in polynomial space.
\end{proof}

For the hardness part, let us first consider a related problem of deciding whether the spectral degeneracy
is greater or equal to some given constant.

\begin{theorem}\label{thm0-hard}
For any fixed integer $d\ge 3$, verifying whether the spectral degeneracy of a graph is at least $d$
is NP-hard, even when restricted to graphs of maximum degree $d+1$.
\end{theorem}

\begin{proof}
We give a reduction from the problem of finding a $d$-regular subgraph in a graph $G$ of
maximum degree $d+1$, which is NP-hard by Theorem~\ref{thm-regsg}.  Let $G'$ be the graph
obtained from $G$ by replacing each edge $uv$ by a graph $G_{uv}$ created from a clique on $d+1$ new vertices
by removing an edge $xy$ and adding the edges $ux$ and $vy$.  Consider a connected subgraph $H\subseteq G'$.  If $H$ is $d$-regular and
$z\in V(H)$ belongs to $V(G_{uv})\setminus \{u,v\}$, then $G_{uv}\subseteq H$.  It follows that
$G'$ contains a $d$-regular subgraph if and only if $G$ contains a $d$-regular subgraph.

Furthermore, if $\Delta(H)=d+1$, then by Lemma~\ref{lemma-near} we have
$$\rho(H)\le \sqrt[3]{(d+1)(d^2+1)}<\sqrt{d\Delta(H)},$$ and if $\Delta(H)\le d$,
then $\rho(H)\le\sqrt{d\Delta(H)}$, where the equality holds if and only if $H$ is $d$-regular.
Therefore, $G$ has a $d$-regular subgraph if and only if the spectral degeneracy of $G'$ is
at least $d$.  Since the size of $G'$ is polynomial in the size of $G$, this shows that
deciding whether the spectral degeneracy of a graph is at least $d$ is NP-hard.
\end{proof}

A small variation of this analysis gives us the desired result.

\begin{theorem}\label{thm-hard}
The {\sc Spectral Degeneracy Problem} is co-NP-complete.
\end{theorem}

\begin{proof}
By Lemma~\ref{lem-coNP}, the problem is in co-NP, so it remains to exhibit a reduction
from a co-NP-hard problem.

Consider the graph $G'$ from the proof of Theorem~\ref{thm0-hard} and its connected subgraph $H$.  
If $H$ has maximum degree $d+1$, then the spectral radius of $H$ is at most 
$$\sqrt{\sqrt[3]{\frac{(d^2+1)^2}{d+1}}\Delta(H)}$$
by Lemma~\ref{lemma-near}.
If $H$ has maximum degree at most $d-1$, then $$\rho(H)\le \sqrt{(d-1)\Delta(H)}.$$  Finally,
if $\Delta(H)=d$ and $H$ is not $d$-regular, then 
$$\rho(H)\le \sqrt{\left(d-|V(H)|^{-2}\right)^2}\le \sqrt{\left(d-|V(H)|^{-2}\right)\Delta(H)}$$
by Lemma~\ref{lemma-nearlw}.

Let $n=|V(G')|$.  Let $b$ be a rational number such that
$$\max\left\{\sqrt[3]{\frac{(d^2+1)^2}{d+1}},\ d-1,\ d-n^{-2}\right\}\le b<d.$$
We conclude that either $G'$ has spectral radius at least $d$ or at most $b$.  Thus, deciding whether
the spectral degeneracy of a graph is at most $b$ (where $b$ is part of the input) is co-NP-hard.
\end{proof}

However, this does not exclude the possibility that the spectral degeneracy could be approximated
efficiently. Let $\varepsilon>0$ be a constant.

\begin{quote}
	{\sc Approximate spectral degeneracy} \\
	{\sc Input}: A graph $G$ and a rational number $d$. \\
	{\sc Task}: Either prove that $G$ is spectrally $(1+\varepsilon)d$-degenerate,
or show that it is not spectrally $d$-degenerate.
\end{quote}

Does there exist $\varepsilon$ such that this problem can be solved in a polynomial time?
Or possibly, is it true that this question can be solved in a polynomial time for every $\varepsilon>0$?
Both of these questions are open.

\bibliographystyle{siam}
\bibliography{SpectralRadiusDegenerate}
\end{document}